\documentclass[12pt]{amsart}
\usepackage{amsfonts, amsmath, amsthm}

\usepackage{graphicx}
\usepackage{epstopdf}
\usepackage{psfrag}

\newtheorem{pro}{Proposition}[section]
\newtheorem{thm}[pro]{Theorem}
\newtheorem{lem}[pro]{Lemma}

\newtheorem{cor}[pro]{Corollary}
\newtheorem{quest}[pro]{Question}

\theoremstyle{definition}
\newtheorem{dfn}[pro]{Definition}

\theoremstyle{remark}

\newcommand{\VV}{\mathcal V}
\newcommand{\WW}{\mathcal W}

\newcommand{\bdy}{\partial}

\newcommand{\thick}[1]{{\rm Thick}(#1)}
\newcommand{\thin}[1]{{\rm Thin}(#1)}

\newcommand{\amlg}[1]{\mathcal A(#1)}

\title{Heegaard splittings of sufficiently complicated 3-manifolds II: Amalgamation} 
\date{\today}
\address{Pitzer College}
\email{bachman@pitzer.edu}
\author{David Bachman}
\begin{document}

\begin{abstract}
Let $M_1$ and $M_2$ be compact, orientable 3-manifolds, and $M$ the manifold obtained by gluing some component $F$ of $\bdy M_1$ to some component of $\bdy M_2$ by a homeomorphism $\phi$. We show that when $\phi$ is ``sufficiently complicated" then (1)  the amalgamation of low genus, unstabilized, boundary-unstabilized Heegaard splittings of $M_i$ is an unstabilized splitting of $M$, (2) every low genus, unstabilized Heegaard splitting of $M$ can be expressed as an amalgamation of unstabilized, boundary-unstabilized splittings of $M_i$, and possibly a Type II splitting of $F \times I$, and (3) if there is no Type II splitting in such an expression then it is unique.  
\end{abstract}

\maketitle

\markright{SPLITTINGS OF SUFFICIENTLY COMPLICATED 3-MANIFOLDS II}

\section{Introduction.}

Given a Heegaard surface in a 3-manifold $M$ one can {\it stabilize} to obtain a splitting of higher genus by taking the connected sum with the genus one splitting of $S^3$. Thus, to understand the set of all splittings of $M$ one should begin with the unstabilized ones. When $M$ is obtained by gluing two other 3-manifolds, $M_1$ and $M_2$, along their boundaries, then an important question is to determine the extent to which unstabilized splittings of $M_1$ and $M_2$ determine the unstabilized splittings of $M$. For example, in Problem 3.91 of \cite{kirby:97}, Cameron Gordon conjectured that the connected sum of  unstabilized Heegaard splittings is unstabilized. This was proved by the author in \cite{gordon}, and by Scharlemann and Qiu in \cite{ScharlemannQiu}. 

Given Heegaard splittings $H_i \subset M_i$, Schultens gave a construction of a Heegaard splitting of $M$, called their {\it amalgamation} \cite{schultens:93}. Using this terminology, we can phrase the higher genus analogue of Gordon's conjecture:

\begin{quest}
\label{c:IsotopyConjecture}
Let $M_1$ and $M_2$ denote compact, orientable, irreducible 3-manifolds with homeomorphic, incompressible boundary. Let $M$ be the 3-manifold obtained from $M_1$ and $M_2$ by gluing their boundaries by some homeomorphism. Let $H_i$ be an unstabilized Heegaard splitting of $M_i$. Is the amalgamation of $H_1$ and $H_2$ in $M$ unstabilized? 
\end{quest}

As stated, Schultens and Weidmann have shown the answer to this question is no \cite{SchultensWeidmann}. In light of their examples we refine the question by adding the hypothesis that the gluing map between $M_1$ and $M_2$ is  ``sufficiently complicated," in some suitable sense. We will postpone a precise definition of this term to Section \ref{s:BarrierSurfaces}. However, throughout this paper it will be used in such a way so that if $\psi:T^1 \to T^2$ is a fixed homeomorphism, then for each Anosov map $\phi:T^2 \to T^2$, there exists an $N$ so that for each $n \ge N$, $\psi^{-1}\phi^n \psi$ is sufficiently complicated. Unfortunately, the assumption that the gluing map of Question \ref{c:IsotopyConjecture} is sufficiently complicated is still not a strong enough hypothesis to insure the answer is yes, as the following construction shows.

If $\bdy M \ne \emptyset$, then one can {\it boundary-stabilize} a Heegaard splitting of $M$ by tubing a copy of a component of $\bdy M$ to it \cite{moriah:02}.  Let $M_1$ be a manifold that has a boundary component $F$, and an unstabilized Heegaard splitting $H_1$ that has been obtained by boundary-stabilizing some other splitting along $F$. (See \cite{sedgwick:01} or \cite{ms:04} for such examples.) Let $M_2$ be a manifold with a boundary component homeomorphic to $F$, and a {\it $\gamma$-primitive} Heegaard splitting (see \cite{moriah:02}). Such a Heegaard splitting is unstabilized, but has the property that boundary-stabilizing it along $F$ produces a stabilized splitting. Then no matter how we glue $M_1$ to $M_2$ along $F$, the amalgamation of $H_1$ and $H_2$ will be stabilized. 

Given this example, and those of Schultens and Weidmann, we deduce the following: In order for the answer to Question \ref{c:IsotopyConjecture} to be yes, we would at least have to know that $H_1$ and $H_2$ are not stabilized, not boundary-stabilized, and that the gluing map is sufficiently complicated. Our main result is that these hypotheses are enough to obtain the desired result:

\medskip

\noindent {\bf Theorem \ref{t:HigherGenusGordon}. }{\it
Let $M_1$ and $M_2$ be compact, orientable, irreducible 3-manifolds with incompressible boundary, neither of which is an $I$-bundle. Let $M$ denote the manifold obtained by gluing some component $F$ of $\bdy M_1$ to some component of $\bdy M_2$ by some homeomorphism $\phi$. Let $H_i$ be an unstabilized, boundary-unstabilized Heegaard splitting of $M_i$. If $\phi$ is sufficiently complicated then the amalgamation of $H_1$ and $H_2$ in $M$ is unstabilized.}

\medskip

This result allows us to construct the first example of a non-minimal genus Heegaard splitting which has Hempel distance \cite{hempel:01} exactly one.  The first examples of minimal genus, distance one Heegaard splittings were found by Lustig and Moriah in 1999  \cite{moriah:99}. Since then the existence of non-minimal genus examples was expected, but a construction remained elusive. This is why Moriah has called the search for such examples the ``nemesis of Heegaard splittings" \cite{moriah}. In some sense they are the last form of Heegaard splitting to be found.  In Corollary \ref{c:Moriah} we produce manifolds that have an arbitrarily large number of such splittings.

The conclusion of Theorem \ref{t:HigherGenusGordon} asserts that each pair of low genus, unstabilized, boundary-unstabilized splittings of $M_1$ and $M_2$ determines an unstabilized splitting of $M_1 \cup _\phi M_2$. We now discuss the converse of this statement. Lackenby \cite{lackenby:04}, Souto \cite{souto}, and Li \cite{li:08} have independently shown that when $\phi$ is sufficiently complicated, then any low genus Heegaard splitting $H$ of $M_1 \cup _\phi M_2$ is an amalgamation of splittings $H_i$ of $M_i$. In Theorem \ref{t:AmalgamationExists} we prove a refinement of this result:

\medskip

\noindent {\bf Theorem \ref{t:AmalgamationExists}. }{\it
Let $M_1$ and $M_2$ be compact, orientable, irreducible 3-manifolds with incompressible boundary, neither of which is an $I$-bundle. Let $M$ denote the manifold obtained by gluing some component $F$ of $\bdy M_1$ to some component of $\bdy M_2$ by some homeomorphism $\phi$. If $\phi$ is sufficiently complicated then any low genus, unstabilized Heegaard splitting of $M$ is an amalgamation of unstabilized, boundary-unstabilized splittings of $M_1$ and $M_2$, and possibly a Type II splitting of $F \times I$.}

\medskip

Here a Type II splitting of $F \times I$ consists of two copies of $F$ connected by an unknotted tube (see \cite{st:93}). Suppose, as in the theorem above, that $F$ is a boundary component of $M_1$, and $H_1$ is a Heegaard splitting of $M_1$. If we glue $F \times I$ to $\bdy M_1$, and amalgamate $H_1$ with a Type II splitting of $F \times I$, then the result is the same as if we had just boundary-stabilized $H_1$. 

Ideally, we would like to say that the splittings of $M_i$ given by Theorem \ref{t:AmalgamationExists} are uniquely determined by the Heegaard splitting of $M$ from which they come. However, no matter how complicated $\phi$ is this may not be the case, as the following construction shows.

Let $M_1$ be a 3-manifold with boundary homeomorphic to a surface $F$, that has inequivalent unstabilized, boundary-unstabilized splittings $H_1$ and $G_1$ that become equivalent after a boundary-stabilization. (For example, $M_1$ may be a Seifert fibered space with a single boundary component. {\it Vertical} splittings $H_1$ and $G_1$ would then be equivalent after a boundary stabilization, by \cite{schultens:96}.) Let $M_2$ be any 3-manifold with boundary homeomorphic to $F$, and let $H_2$ be an unstabilized, boundary-unstabilized Heegaard splitting of $M_2$. Glue $M_1$ to $M_2$ by any map $\phi$ to create the manifold $M$. Let $H$ be the amalgamation of $H_1$, $H_2$, and a Type II splitting of $F \times I$. Then $H$ is also the amalgamation of $G_1$, $H_2$ and a Type II splitting of $F \times I$. So the expression of $H$ as an amalagamation as described by the conclusion of Theorem \ref{t:AmalgamationExists} is not unique. 

This construction shows that Type II splittings of $F \times I$ are obstructions to the uniqueness of the decomposition given by Theorem \ref{t:AmalgamationExists}. In our final theorem, we show that this is the only obstruction:

\medskip

\noindent {\bf Theorem \ref{t:HighGenusGordonIsotopy}. }{\it
Let $M_1$ and $M_2$ be compact, orientable, irreducible 3-manifolds with incompressible boundary, neither of which is an $I$-bundle. Let $M$ denote the manifold obtained by gluing some component $F$ of $\bdy M_1$ to some component of $\bdy M_2$ by some homeomorphism $\phi$. Suppose $\phi$ is sufficiently complicated, and some low genus Heegaard splitting  $H$ of $M$ can be expressed as an amalgamation of unstabilized, boundary-unstabilized splittings of $M_1$ and $M_2$. Then this expression is unique.}

\medskip

This paper relies heavily on the technical machinery developed in \cite{gordon}. In Sections 2 through 5 we review this material. In Section 6 we state an important result from \cite{StabilizationResults}, which follows from the main result of \cite{barrier}. Anyone who has read \cite{StabilizationResults} can skip directly to Sections 7, 8, and 9, where we establish the results described above.

The author thanks Ryan Derby-Talbot for several helpful comments.

\section{Heegaard and Generalized Heegaard Splittings}

In this section we define {\it Heegaard splittings} and {\it Generalized Heegaard Splitting}. The latter structures were first introduced by Scharlemann and Thompson \cite{st:94} as a way of keeping track of handle structures. The definition we give here is more consistent with the usage in \cite{gordon}.

\begin{dfn}
A {\it compression body} $\mathcal C$ is a manifold formed in one of the following two ways:
	\begin{enumerate}
		\item Starting with a 0-handle, attach some number of 1-handles. In this case we say $\bdy _- \mathcal C=\emptyset$ and $\bdy _+ \mathcal C=\bdy \mathcal C$. 
		\item Start with some (possibly disconnected) surface $F$ such that each component has positive genus. Form the product $F \times I$. Then attach some number of 1-handles to $F \times \{1\}$. We say $\bdy _- \mathcal C= F \times \{0\}$ and $\bdy _+ \mathcal C$ is the rest of $\bdy \mathcal C$. 
	\end{enumerate}
\end{dfn}

\begin{dfn}
\label{d:Heegaard}
Let $H$ be a properly embedded, transversally oriented surface in a 3-manifold $M$, and suppose $H$ separates $M$ into $\VV$ and $\WW$. If $\VV$ and $\WW$ are compression bodies and $\VV \cap \WW=\partial _+ \VV=\partial _+ \WW=H$, then we say $H$ is a {\it Heegaard surface} in $M$.
\end{dfn}

\begin{dfn}
The transverse orientation on the Heegaard surface $H$ in the previous definition is given by a choice of normal vector. If this vector points into $\VV$, then we say any subset of $\VV$ is {\it above} $H$ and any subset of $\WW$ is {\it below} $H$. 
\end{dfn}

\begin{dfn}
\label{d:GHS}
A {\it generalized Heegaard splitting (GHS)} $H$ of a 3-manifold $M$ is a pair of sets of transversally oriented, connected, properly embedded surfaces,  $\thick{H}$ and $\thin{H}$ (called the {\it thick levels} and {\it thin levels}, respectively), which satisfy the following conditions. 
	\begin{enumerate}
		\item Each component $M'$ of $M \setminus \thin{H}$ meets a unique element $H_+$ of $\thick{H}$. The surface $H_+$ is a Heegaard surface in $\overline{M'}$ dividing $\overline{M'}$ into compression bodies $\VV$ and $\WW$. Each component of $\bdy _- \VV$ and $\bdy _- \WW$ is an element of $\thin{H}$. Henceforth we will denote the closure of the component of $M \setminus \thin{H}$ that contains an element $H_+ \in \thick{H}$ as $M(H_+)$. 
		\item Suppose $H_-\in \thin{H}$. Let $M(H_+)$ and $M(H_+')$ be the submanifolds on each side of $H_-$. Then $H_-$ is below $H_+$ in $M(H_+)$ if and only if it is above $H_+'$ in $M(H_+')$.
		\item The term ``above" extends to a partial ordering on the elements of $\thin{H}$ defined as follows. If $H_-$ and $H'_-$ are subsets of $\bdy M(H_+)$, where $H_-$ is above $H_+$ in $M(H_+)$ and $H_-'$ is below $H_+$ in $M(H_+)$, then $H_-$ is above $H_-'$ in $M$.
	\end{enumerate}
\end{dfn}

\section{Reducing GHSs}

\begin{dfn}
Let $H$ be an embedded surface in $M$.  Let $D$ be a compression for $H$. Let $\VV$ denote the closure of the component of $M \setminus H$ that contains $D$. (If $H$ is non-separating then $\VV$ is the manifold obtained from $M$ by cutting open along $H$.) Let $N$ denote a regular neighborhood of $D$ in $\VV$. To {\it surger} or {\it compress} $H$ along $D$ is to remove $N \cap H$ from $H$ and replace it with the frontier of $N$ in $\VV$. We denote the resulting surface by $H/D$. 
\end{dfn}

It is not difficult to find a complexity for surfaces which decreases under compression. We now present an operation that one can perform on GHSs that also reduces some complexity (see Lemma 5.14 of \cite{gordon}).  This operation is called {\it weak reduction}. 

\begin{dfn}
Let $H$ be a separating, properly embedded surface in $M$. Let $D$ and $E$ be compressions on opposite sides of $H$. Then we say $(D,E)$ is a {\it weak reducing pair} for $H$ if $D \cap E=\emptyset$. When $(D,E)$ is a weak reducing pair, then we let $H/DE$ denote the result of simultaneous surgery along $D$ and $E$.
\end{dfn}

\begin{dfn}
\label{d:PreWeakReduction}
Let $M$ be a compact, connected, orientable 3-manifold. Let $G$ be a GHS. Let $(D,E)$ be a weak reducing pair for some  $G_+ \in \thick{G}$. Define 
	\[T(H)=\thick{G} -\{G_+\} \cup \{G_+ / D, G_+ / E\},\ \mbox{and}\]
	\[t(H) =\thin{G} \cup \{G_+/DE\}.\]

A new GHS $H=\{\thick{H},\thin{H}\}$ is then obtained from $\{T(H), t(H)\}$ by successively removing the following:
	\begin{enumerate}
		\item Any sphere element $S$ of $T(H)$ or $t(H)$ that is inessential, along with any elements of $t(H)$ and $T(H)$ that lie in the ball that it bounds. 
		\item Any element $S$ of $T(H)$ or $t(H)$ that is $\bdy$-parallel, along with any elements of $t(H)$ and $T(H)$ that lie between $S$ and $\bdy M$. 
		\item Any elements $H_+ \in T(H)$ and $H_- \in t(H)$, where $H_+$ and $H_-$ cobound a submanifold $P$ of $M$, such that $P$ is a product, $P \cap T(H)=H_+$, and $P \cap t(H)=H_-$. 
	\end{enumerate}

We say the GHS $H$ is obtained from $G$ by {\it weak reduction} along $(D,E)$.
\end{dfn}

The first step in weak reduction is illustrated in Figure \ref{f:WeakReduction}. 

        \begin{figure}[htbp]
        \psfrag{1}{$G_+/D$}
        \psfrag{2}{$G_+/E$}
        \psfrag{3}{$G_+/DE$}
        \psfrag{G}{$G_+$}
        \psfrag{E}{$E$}
        \psfrag{D}{$D$}
        \vspace{0 in}
        \begin{center}
       \includegraphics[width=3.5 in]{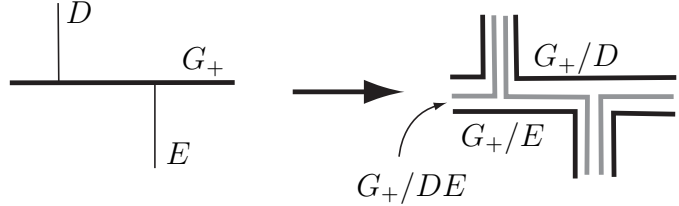}
       \caption{The first step in weak reduction.}
        \label{f:WeakReduction}
        \end{center}
        \end{figure}

\begin{dfn}
\label{d:destabilization}
The weak reduction of a GHS given by the weak reducing pair $(D,E)$ for the thick level $G_+$ is called a {\it destabilization} if $G_+/DE$ contains a sphere.
\end{dfn}

\begin{dfn}
Suppose $H$ is a Heegaard splitting of a manifold $M$ with non-empty boundary. Let $F$ denote a component of $\bdy M$. Then the surface $H'$ obtained from $H$ by attaching a copy of $F$ to it by an unknotted tube is also a Heegaard surface in $M$. We say $H'$ was obtained from $H$ by a {\it boundary-stabilization along $F$}. The reverse operation is called a {\it boundary-destabilization} along $F$. 
\end{dfn}

\begin{dfn}
Suppose $H$ is a GHS of $M$. Let $N$ denote a submanifold of $M$ bounded by elements of $\thin{H}$. Then we may define a GHS $H(N)$ of $N$. The thick and thin levels of $H(N)$ are the thick and thin levels of $H$ that lie in $N$. 
\end{dfn}

\section{Amalgamation}

Let $H$ be a GHS of a connected 3-manifold $M$. In \cite{gordon} we use $H$ to produce a complex that is the spine of a Heegaard splitting of $M$. We call this splitting the {\it amalgamation} of $H$, and denote it $\amlg{H}$. This splitting is defined in such a way so that if $H$ has a unique thick level $H_+$, then $\amlg{H}=H_+$. 

\begin{lem}
\label{l:AmalgGenus}
{\rm (\cite{gordon}, Corollary 7.5)}  Suppose $M$ is irreducible, $H$ is a GHS of $M$ and $G$ is obtained from $H$ by a weak reduction which is not a destabilization. Then $\amlg{H}$ is isotopic to $\amlg{G}$.
\end{lem}

Every GHS $G$ comes from some GHS $H$ with a single thick level $H_+$ by a sequence of weak reductions that are not stabilizations. By Lemma \ref{l:AmalgGenus}, it follows that $\amlg{G}=H_+$. It also follows that if a GHS $G$ is obtained from a GHS $H$ by a weak reduction or a destabilization then the genus of $\amlg{G}$ is at most the genus of $\amlg{H}$. 

\begin{dfn}
The {\it genus} of a GHS is the genus of its amalgamation. 
\end{dfn}

\section{Sequences of GHSs}
\label{s:LastDefSection}

\begin{dfn}
A {\it Sequence Of GHSs} (SOG), $\{H^i\}$ of $M$ is a finite sequence such that for each $i$ either $H^i$ or $H^{i+1}$ is obtained from the other by a weak reduction.
\end{dfn}

\begin{dfn}
If $\bf H$ is a SOG and $k$ is such that $H^{k-1}$ and $H^{k+1}$ are obtained from $H^k$ by a weak reduction then we say the GHS $H^k$ is {\it maximal} in $\bf H$. 
\end{dfn}

It follows that maximal GHSs are larger than their immediate predecessor and immediate successor.

\begin{dfn}
The {\it genus} of a SOG is the maximum among the genera of its GHSs.
\end{dfn}

Just as there are ways to make a GHS ``smaller," there are also ways to make a SOG ``smaller." These are called {\it SOG reductions}, and are explicitly defined in Section 8 of \cite{gordon}. If the first and last GHS of a SOG admit no weak reductions, and there are no SOG reductions then the SOG is said to be {\it irreducible}. For our purposes, all we need to know about SOG reduction is that the maximal GHSs of the new SOG are obtained from the maximal GHSs of the old one by weak reduction, and the following lemma holds:

\begin{lem}
\label{l:GenusGoesDown}
If a SOG $\Lambda$ is obtained from an SOG $\Gamma$ by a reduction then the genus of $\Gamma$ is at least the genus of $\Lambda$. 
\end{lem}

\begin{proof}
Since weak reduction can only decrease the genus of a GHS, the genus of a SOG is the maximum among the genera of its maximal GHSs. But if one SOG is obtained from another by a reduction, then its maximal GHSs are obtained from GHSs of the original by weak reductions. The result thus follows from Lemma \ref{l:AmalgGenus}.
\end{proof}

\section{Barrier surfaces}
\label{s:BarrierSurfaces}

The following definition is different than the one given in \cite{StabilizationResults}, but by Lemma 7.4 of that paper they are equivalent. 

\begin{dfn}
An incompressible surface $F$ in a 3-manifold $M$ is a {\it $g$-barrier surface} if $F$ is isotopic to a thin level of every element of every irreducible SOG of $M$ whose genus is at most $g$. 
\end{dfn}

Note that a single element, irreducible SOG is a GHS that admits no weak reductions. It follows that if $H$ is such a GHS whose genus is at most $g$, and $F$ is a $g$-barrier surface, then $F$ is isotopic to a thin level of $H$. 

We now discuss the existence of $g$-barrier surfaces. For the remainder of this section, let $M$ be a compact, irreducible, (possibly disconnected) 3-manifold with incompressible boundary, such that no component of $M$ is an $I$-bundle. Suppose boundary components $F_1$ and $F_2$ of $M$ are homeomorphic.  Let $M_\phi$ be the manifold obtained from $M$ by gluing these boundary components together by the map $\phi:F_1 \to F_2$. 

Let $Q$ denote a properly embedded (possibly disconnected) surface in $M$ of maximal Euler characteristic, which is both incompressible and $\bdy$-incompressible, and is incident to both $F_1$ and $F_2$.  Then we define the {\it distance} of $\phi$ to be the distance between the loops of $\phi(F_1 \cap Q)$ and $F_2 \cap Q$. When the genus of $F_2$ is at least two, then this distance is measured in the curve complex of $F_2$. If $F_2 \cong T^2$, then this distance is measured in the Farey graph. 

The following is Lemma 7.4 of \cite{StabilizationResults}. It is a direct consequence of the main result of \cite{barrier}.

\begin{thm}
\label{t:Barrier}
Let $F$ denote the image of $F_1$ in $M_\phi$. There is a constant $K$, depending linearly on $\chi(Q)$, such that if the distance of $\phi \ge Kg$, then $F$ is a $g$-barrier surface in $M_\phi$.
\end{thm}

By employing Theorem \ref{t:Barrier} we may construct 3-manifolds with any number of $g$-barrier surfaces. Simply begin with a collection of 3-manifolds and successively glue boundary components together by sufficiently complicated maps.

\section{Amalgamations of unstabilized Heegaard splittings}
\label{s:Amalgamations}

\begin{thm}
\label{t:HigherGenusGordon}
Let $M_1$ and $M_2$ be compact, orientable, irreducible 3-manifolds with incompressible boundary, neither of which is an $I$-bundle. Let $M$ denote the manifold obtained by gluing some component $F$ of $\bdy M_1$ to some component of $\bdy M_2$ by some homeomorphism $\phi$. Let $H_i$ be an unstabilized, boundary-unstabilized Heegaard splitting of $M_i$. If $\phi$ is sufficiently complicated then the amalgamation of $H_1$ and $H_2$ in $M$ is unstabilized.
\end{thm}

Here the term ``sufficiently complicated" means that the distance of $\phi$ is high enough so that by Theorem \ref{t:Barrier} the surface $F$ becomes a $g$-barrier surface, where $g=\mbox{genus}(H_1)+\mbox{genus}(H_2)-\mbox{genus}(F)$.

\begin{proof}
Let ${\bf \Gamma}$ be the SOG depicted in Figure \ref{f:InitialSOGgordon}. The second GHS pictured is the one whose thick levels are $H_1$ and $H_2$. The first GHS in the figure is obtained from this one by a maximal sequence of weak reductions. The third GHS is the one whose only thick level is the amalgamation $H$ of $H_1$ and $H_2$. The next GHS pictured is obtained from $H$ by some number of destabilizations. Finally, the last GHS is obtained from the second to last by a maximal sequence of weak reductions. Note that by construction, $\mbox{genus}({\bf \Gamma})=\mbox{genus}(H)=g$. (For the second equality see, for example, Lemma 5.7 of \cite{StabilizationResults}.)

        \begin{figure}[htbp]
        \psfrag{X}{$H_2$}
        \psfrag{H}{$H_1$}
        \psfrag{F}{$F$}
        \psfrag{1}{$H$}
        \psfrag{2}{$G$}
        \vspace{0 in}
        \begin{center}
       \includegraphics[width=3.5 in]{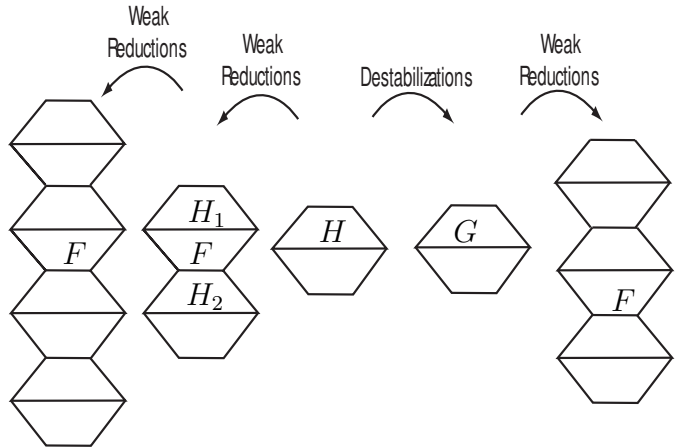}
       \caption{The initial SOG, $\bf \Gamma$.}
        \label{f:InitialSOGgordon}
        \end{center}
        \end{figure}

Now let ${\bf \Lambda}=\{\Lambda ^i\}_{i=1}^n$ be the SOG obtained from ${\bf \Gamma}$ by a maximal sequence of SOG reductions. When the first and last GHS of a SOG admit no weak reductions, then they remain unaffected by SOG reduction. Hence, $\Lambda^1$ is the first element of ${\bf \Gamma}$ and $\Lambda ^n$ is the last element of ${\bf \Gamma}$.

Since $F$ is a $g$-barrier surface, it is isotopic to a thin level of every GHS of $\bf \Lambda$. Let $m$ denote the largest number such that $F$ is isotopic to a {\it unique} thin level $F_i$ of $\Lambda ^i$, for all $i \le m$. The surface $F_i$ then divides $M$ into manifolds $M_1^i$ and $M_2^i$, homeomorphic to $M_1 $ and $M_2$, for each $i \le m$. 

Now note that there are no stabilizations in the original SOG $\bf \Gamma$. It thus follows from Lemma 8.12 of \cite{gordon} that the first destabilization in $\bf \Lambda$ happens before the first stabilization. Furthermore, as the genus of $\Lambda ^n$ is less than the genus of $\Lambda ^1$, there is at least one destabilization in $\bf \Lambda$. Let $p$ denote the smallest value for which $\Lambda^{p+1}$ is obtained from $\Lambda ^p$ by a destabilization. Then for all $i\le p$, either $\Lambda^i$ or $\Lambda ^{i-1}$ is obtained from the other by a weak reduction that is not a destabilization. 

If $p \le m$, then for all $i \le p$, either $\Lambda^i(M_1^i)=\Lambda^{i-1}(M_1^{i-1})$ or one of $\Lambda^i(M_1^i)$ and $\Lambda^{i-1}(M_1^{i-1})$ is obtained from the other by a weak reduction that is not a destabilization. 
It follows from Lemma \ref{l:AmalgGenus} that $H_1^i=\amlg{\Lambda^i(M_1^i)}$ is the same for all $i\le p$. But $H_1^1=H_1$, so $H_1^p=H_1$. By identical reasoning $H_2^p=\amlg{\Lambda^p(M_2^p)}=H_2$. But $H_1$ and $H_2$ are unstabilized, so neither $H_1^{p+1}$ nor $H_2^{p+1}$  can be obtained from $H_1^p$ or $H_2^p$ by destabilization, a contradiction. 

We thus conclude $p > m$, and thus $H_1^m=H_1$ and $H_2^m=H_2$. In particular, it follows that $m$ is strictly less than $n$. That is, there exists a GHS $\Lambda ^{m+1}$ which has two thin levels isotopic to $F$. 

Since $\Lambda ^{m+1}$ has a thin level that is not a thin level of $\Lambda^m$, it must be obtained from $\Lambda^m$ by a weak reduction. It follows that there is some thin level $F_{m+1}$ of $\Lambda ^{m+1}$ that is identical to $F_m$. The other thin level of $\Lambda ^{m+1}$ that is isotopic to $F$ we call $F_{m+1}'$. The surface $F_{m+1}'$ either lies in $M_1^m$ or $M_2^m$. Assume the former. Let $M^{m+1}_1$ denote the side of $F_{m+1}$ homeomorphic to $M_1$. It follows that $\Lambda^{m+1}(M_1^{m+1})$ is obtained from $\Lambda^m(M^m_1)$ by a weak reduction that is not a destabilization. Thus, by  Lemma \ref{l:AmalgGenus}, \[H^{m+1}_1=\amlg{\Lambda^{m+1}(M^{m+1}_1)}=  \amlg{\Lambda^m(M^m_1)}=H_1^m=H_1.\] 

The surfaces $F_{m+1}$ and $F_{m+1}'$ cobound a product region $P$ of $M$. A GHS of $P$ is given by $\Lambda^{m+1}(P)$, and thus $H_P=\amlg{\Lambda^{m+1}(P)}$ is a Heegaard splitting of a product. If this splitting is stabilized, then $H^{m+1}_1$ would be stabilized. But since $H^{m+1}_1=H_1$, and $H_1$ is unstabilized, this is not the case. 

We conclude $H_P$ is an unstabilized Heegaard splitting of $P$. By \cite{st:93} such a splitting admits no weak reductions, and thus  $H_P$ must be the unique thick level of $\Lambda^{m+1}(P)$. From \cite{st:93} this splitting is either  a copy of $F$, or two copies of $F$ connected by a single unknotted tube. In the former case we have a contradiction, as the thick level of  $\Lambda^{m+1}(P)$ would be parallel to the two thin levels $F_{m+1}$ and $F_{m+1}'$, and would thus have been removed during weak reduction. In the latter case $H^{m+1}_1$ is boundary-stabilized. As this Heegard splitting is $H_1$, which is not boundary-stabilzed, we again have a contradiction. 
\end{proof}

A separating surface $H$ in a 3-manifold is said to be {\it weakly reducible} \cite{cg:87} if there is a weak reducing pair for it, and {\it strongly irreducible} otherwise. Note that every Heegaard splitting surface that is an amalgamation of a GHS with multiple thick levels is weakly reducible. 

An example of a 3-manifold that has a weakly reducible, yet unstabilized Heegaard splitting which is not a minimal genus splitting has been elusive. In the next corollary we use Theorem \ref{t:HigherGenusGordon} to construct manifolds that have arbitrarily many such splittings. 

\begin{cor}
\label{c:Moriah}
There exist manifolds that contain arbitrarily many non-minimal genus, unstabilized Heegaard splittings which are not strongly irreducible.
\end{cor}

\begin{proof}
Let $M$ denote a 3-manifold with torus boundary, and strongly irreducible Heegaard splittings of arbitrarily high genus. (Such an example has been constructed by Casson and Gordon. See \cite{sedgwick:97}. The manifold they construct is closed, but there is a solid torus that is a core of one of the handlebodies bounded by each Heegaard surface. Thus, removing this solid torus produces  a manifold with torus boundary that has arbitrarily high genus strongly irreducible Heegaard splittings.) 

Now let $M_1$ and $M_2$ be two copies of $M$, and let $H_g^i$ denote a genus $g$ strongly irreducible splitting in $M_i$. As $H_g^i$ is strongly irreducible, it is neither stabilized nor boundary-stabilized. Hence, if $M_1$ is glued to $M_2$ by a sufficiently complicated homeomorphism, it follows from Theorem \ref{t:HigherGenusGordon} that the amalgamation of $H_g^1$ and $H_g^2$ is unstabilized, for all $g \le G$. (One can make $G$ as high as desired without changing the genus of $M_1 \cup M_2$ by gluing $M_1$ to $M_2$ by more and more complicated maps.)

Finally, note that every amalgamation is weakly reducible. 
\end{proof}

\section{Low genus splittings are amalgamations}

In this section we establish a refinement of a result due independently to Lackenby \cite{lackenby:04}, Souto \cite{souto}, and Li \cite{li:08}. Their result says that if 3-manifolds $M_1$ and $M_2$ are glued by a sufficiently complicated map, then all low genus, unstabilized Heegaard splittings of the resulting manifold are amalgamations of splittings of $M_1$ and $M_2$. 

\begin{thm}
\label{t:AmalgamationExists}
Let $M_1$ and $M_2$ be compact, orientable, irreducible 3-manifolds with incompressible boundary, neither of which is an $I$-bundle. Let $M$ denote the manifold obtained by gluing some component $F$ of $\bdy M_1$ to some component of $\bdy M_2$ by some homeomorphism $\phi$. If $\phi$ is sufficiently complicated then any low genus, unstabilized Heegaard splitting $H$ of $M$ is an amalgamation of unstabilized, boundary-unstabilized splittings of $M_1$ and $M_2$, and possibly a Type II splitting of $F \times I$.
\end{thm}

Here the terms ``sufficiently complicated" and ``low genus" mean that the distance of $\phi$ is high enough so that by Theorem \ref{t:Barrier} the surface $F$ becomes a $g$-barrier surface, where $g=\mbox{genus}(H)$.

\begin{proof}
Let $H_*$ be an unstabilized Heegaard splitting of $M$ whose genus is at most $g$. Let $H$ be a GHS obtained from the GHS whose only thick level is $H_*$ by a maximal sequence of weak reductions (Figure \ref{f:AmalgamExists}(b)).  Since $H_*$ was unstabilized, it follows from Lemma \ref{l:AmalgGenus} that $\amlg{H}=H_*$. 

        \begin{figure}[htbp]
        \psfrag{H}{$H_*$}
        \psfrag{F}{$F$}
        \psfrag{a}{$G_1$}
        \psfrag{b}{$G_2$}
        \psfrag{g}{$H_1$}
        \psfrag{G}{$H_2$}
        \psfrag{h}{$H_F$}
        \psfrag{1}{(a)}
        \psfrag{2}{(b)}
        \psfrag{3}{(c)}
        \psfrag{4}{(d)}
        \psfrag{5}{(e)}
        \vspace{0 in}
        \begin{center}
       \includegraphics[width=3.5 in]{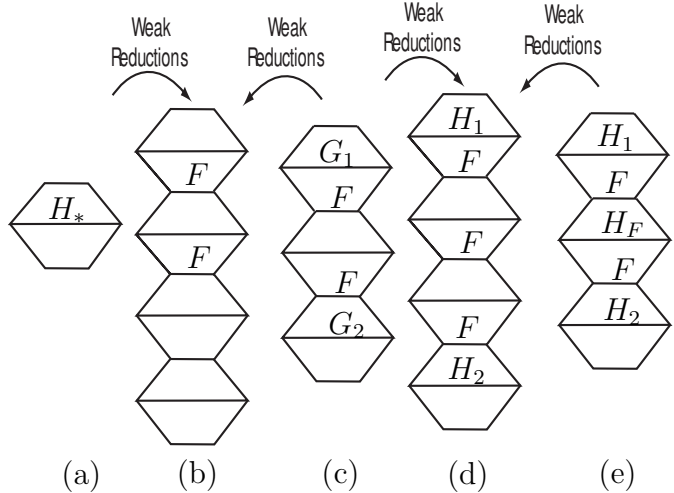}
       \caption{The GHSs of the proof of Theorem \ref{t:AmalgamationExists}.}
        \label{f:AmalgamExists}
        \end{center}
        \end{figure}

By Theorem \ref{t:Barrier}, $F$ is a $g$-barrier surface. Hence, $F$ is isotopic to at least one thin level of $H$. Now cut $M$ along all thin levels isotopic to $F$. The result is manifolds $M_1'$ and $M_2'$ homeomorphic to $M_1$ and $M_2$, and possibly several manifolds homeomorphic to $F \times I$. The Heegaard splitting $H_*=\amlg{H}$ is thus an amalgamation of the splittings $G_1=\amlg{H(M_1')}$ and $G_2=\amlg{H(M_2')}$, and possibly a Heegaard splitting of $F \times I$ (Figure \ref{f:AmalgamExists}(c)).

Since $H_*$ is unstabilized, it follows that both $G_1$ and $G_2$ are unstabilized. Now suppose that $G_i$ is boundary-stabilized. Then $G_i$ is the amalgamation of an unstabilized, boundary-unstabilized splitting $H_i$ of $M_i'$, and a splitting of $F \times I$. If $G_i$ was boundary-unstabilized to begin with, then let $H_i=G_i$. Thus, $H_*$ is an amalgamation of $H_1$, $H_2$, and possibly multiple splittings of $F \times I$ (Figure \ref{f:AmalgamExists}(d)), which can again be amalgamated to a single splitting $H_F$ of $F \times I$ (Figure \ref{f:AmalgamExists}(e)). 

By \cite{st:93} $H_F$ is a stabilization of either a copy of $F$ (i.e. a stabilization of a Type I splitting), or of two copies of $F$ connected by a vertical tube (i.e. a stabilization of a Type II splitting). However, our assumption that $H_*$ was unstabilized implies $H_F$ is unstabilized. Furthermore, as $H_F$ comes from amalgamating non-trivial splittings of $F \times I$,  it will not be a Type I splitting. We conclude that 
the only possibility is that $H_F$ is a Type II splitting of $F \times I$. 
\end{proof}

\section{Isotopic Heegaard splittings in amalgamated 3-manifolds.}
\label{s:Isotopy}

In Theorem \ref{t:AmalgamationExists} we showed that when $\phi$ is sufficiently complicated then any low genus, unstabilized Heegaard splitting $H$ of $M_1 \cup _\phi M_2$ is an amalgamation of unstabilized, boundary unstabilized splittings $H_1$ and $H_2$ of $M_1$ and $M_2$, and possibly a Type II splitting of $\bdy M_1 \times I$. In the next theorem we show that when there is no Type II splitting in this decomposition, then $H_1$ and $H_2$ are completely determined by $H$.

\begin{thm}
\label{t:HighGenusGordonIsotopy}
Let $M_1$ and $M_2$ be compact, orientable, irreducible 3-manifolds with incompressible boundary, neither of which is an $I$-bundle. Let $M$ denote the manifold obtained by gluing some component $F$ of $\bdy M_1$ to some component of $\bdy M_2$ by some homeomorphism $\phi$. Suppose $\phi$ is sufficiently complicated, and some low genus Heegaard splitting  $H$ of $M$ can be expressed as an amalgamation of unstabilized, boundary-unstabilized splittings of $M_1$ and $M_2$. Then this expression is unique.
\end{thm}

As in Theorem \ref{t:AmalgamationExists}, the terms ``sufficiently complicated" and ``low genus" mean that the distance of $\phi$ is high enough so that by Theorem \ref{t:Barrier} the surface $F$ becomes a $g$-barrier surface, where $g=\mbox{genus}(H)$.

\begin{proof}
Suppose $H$ can be expressed as an amalgamation of unstabilized, boundary-unstabilized splittings $H_1$ and $H_2$ of $M_1$ and $M_2$. Suppose also $H$ can be expressed as an amalgamation of unstabilized, boundary-unstabilized splittings $G_1$ and $G_2$ of $M_1$ and $M_2$. 

Let ${\bf \Gamma}$ be the SOG depicted in Figure \ref{f:InitialSOGisotopy}. The third GHS in the figure is the one whose only thick level is $H$. The second GHS pictured is the GHS whose thick levels are $H_1$ and $H_2$. The first GHS in the figure is obtained from this one by a maximal sequence of weak reductions. The fourth GHS is the one whose thick levels are $G_1$ and $G_2$. Finally, the last GHS is obtained from the fourth by a maximal sequence weak reductions.

        \begin{figure}[htbp]
        \psfrag{X}{$H_2$}
        \psfrag{H}{$H_1$}
        \psfrag{x}{$G_2$}
        \psfrag{h}{$G_1$}
        \psfrag{G}{$H$}
        \psfrag{F}{$F$}
        \vspace{0 in}
        \begin{center}
       \includegraphics[width=3.5 in]{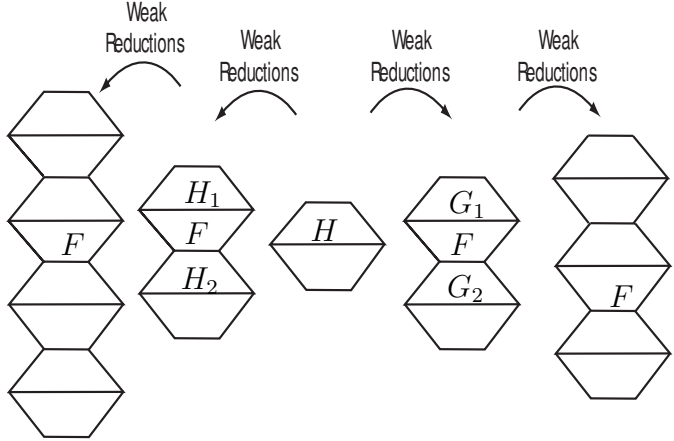}
       \caption{The initial SOG, $\bf \Gamma$.}
        \label{f:InitialSOGisotopy}
        \end{center}
        \end{figure}

Now let ${\bf \Lambda}=\{\Lambda ^i\}_{i=1}^n$ be the SOG obtained from ${\bf \Gamma}$ by a maximal sequence of SOG reductions. When the first and last GHS of a SOG admit no weak reductions, then they remain unaffected by SOG reduction. Hence, $\Lambda^1$ is the first element of ${\bf \Gamma}$ and $\Lambda ^n$ is the last element of ${\bf \Gamma}$.

Note that every GHS of $\bf \Gamma$ is obtained from $H$ by a sequence of weak reductions. By Theorem \ref{t:HigherGenusGordon} the splitting $H$ is unstabilized, and thus every GHS of $\bf \Gamma$ is unstabilized. Furthermore, every GHS of $\bf \Lambda$ is obtained from GHSs of $\bf \Gamma$ by weak reductions. Hence, every GHS of $\bf \Lambda$ is unstabilized. It follows that there are no destabilizations in $\bf \Lambda$. 

Since $F$ is a $g$-barrier surface, it is isotopic to a thin level of every GHS of $\bf \Lambda$. If, for some $i$, we assume the surface $F$ is isotopic to two elements of $\thin{\Lambda^i}$, then the argument given in the proof of Theorem \ref{t:HigherGenusGordon} provides a contradiction. (This is where we use the assumption that $H_1$ and $H_2$ are not boundary-stabilized.)

We conclude, then, that for each $i$ either $\Lambda^i$ or $\Lambda ^{i+1}$ is obtained from the other by a weak reduction that is not a destabilization. Furthermore, since for all $i$ the surface $F$ is isotopic to a unique thin level of $\Lambda ^i$, it follows that for each $i$,  $M_1(\Lambda^i)=M_1(\Lambda ^{i+1})$, or either $M_1(\Lambda^i)$ or $M_1(\Lambda ^{i+1})$ is obtained from the other by a weak reduction that is not a destabilization. It thus follows from Lemma \ref{l:AmalgGenus} that for each $i$ the surface $\amlg{M_1(\Lambda ^i)}$ is the same (up to isotopy). But $\amlg{M_1(\Lambda ^1)}=H_1$ and $\amlg{M_1(\Lambda ^n)}=G_1$. Hence, $H_1$ is isotopic to $G_1$. A symmetric argument shows $H_2$ must be isotopic to $G_2$, completing the proof. 
\end{proof}

\bibliographystyle{alpha}

\begin{thebibliography}{Hem01}

\bibitem[Baca]{barrier}
D.~Bachman.
\newblock Barriers to topologically minimal surfaces.
\newblock Available at {\tt http://arxiv.org/abs/0903.1692}.

\bibitem[Bacb]{StabilizationResults}
D.~Bachman.
\newblock Heegaard splittings of sufficiently complicated 3-manifolds {I}:
  {S}tabilization.
\newblock Available at {\tt http://arxiv.org/abs/0903.1695}.

\bibitem[Bac08]{gordon}
D.~Bachman.
\newblock Connected {S}ums of {U}nstabilized {H}eegaard {S}plittings are
  {U}nstabilized.
\newblock {\em Geometry \& Topology}, (12):2327--2378, 2008.

\bibitem[CG87]{cg:87}
A.~J. Casson and C.~McA. Gordon.
\newblock Reducing {H}eegaard splittings.
\newblock {\em Topology and its Applications}, 27:275--283, 1987.

\bibitem[Hem01]{hempel:01}
J.~Hempel.
\newblock 3-manifolds as viewed from the curve complex.
\newblock {\em Topology}, 40:631--657, 2001.

\bibitem[Kir97]{kirby:97}
Rob Kirby.
\newblock Problems in low-dimensional topology.
\newblock In {\em Geometric topology (Athens, GA, 1993)}, volume~2 of {\em
  AMS/IP Stud. Adv. Math.}, pages 35--473. Amer. Math. Soc., Providence, RI,
  1997.

\bibitem[Lac04]{lackenby:04}
Marc Lackenby.
\newblock The {H}eegaard genus of amalgamated 3-manifolds.
\newblock {\em Geom. Dedicata}, 109:139--145, 2004.

\bibitem[Li]{li:08}
T.~Li.
\newblock Heegaard surfaces and the distance of amalgamation.
\newblock Preprint.

\bibitem[LM99]{moriah:99}
Martin Lustig and Yoav Moriah.
\newblock Closed incompressible surfaces in complements of wide knots and
  links.
\newblock {\em Topology Appl.}, 92(1):1--13, 1999.

\bibitem[Mor]{moriah}
Y.~Moriah.
\newblock Heegaard splittings of knot exteriors.
\newblock Preprint. Available at {\tt http://arxiv.org/abs/math/0608137}.

\bibitem[Mor02]{moriah:02}
Yoav Moriah.
\newblock On boundary primitive manifolds and a theorem of {C}asson-{G}ordon.
\newblock {\em Topology Appl.}, 125(3):571--579, 2002.

\bibitem[MS04]{ms:04}
Yoav Moriah and Eric Sedgwick.
\newblock Closed essential surfaces and weakly reducible {H}eegaard splittings
  in manifolds with boundary.
\newblock {\em J. Knot Theory Ramifications}, 13(6):829--843, 2004.

\bibitem[Sch93]{schultens:93}
Jennifer Schultens.
\newblock The classification of {H}eegaard splittings for (compact orientable
  surface){$\,\times\, S\sp 1$}.
\newblock {\em Proc. London Math. Soc. (3)}, 67(2):425--448, 1993.

\bibitem[Sch96]{schultens:96}
J.~Schultens.
\newblock The stabilization problem for {H}eegaard splittings of {S}eifert
  fibered spaces.
\newblock {\em Topology and its Applications}, 73:133 -- 139, 1996.

\bibitem[Sed97]{sedgwick:97}
E.~Sedgwick.
\newblock An infinite collection of {H}eegaard splttings that are equivalent
  after one stabilization.
\newblock {\em Math. Ann.}, 308:65--72, 1997.

\bibitem[Sed01]{sedgwick:01}
Eric Sedgwick.
\newblock Genus two 3-manifolds are built from handle number one pieces.
\newblock {\em Algebr. Geom. Topol.}, 1:763--790 (electronic), 2001.

\bibitem[Sou]{souto}
Juan Souto.
\newblock Distances in the curve complex and {H}eegaard genus.
\newblock Preprint. Available at {\tt www.picard.ups-tlse.fr/\%7Esouto/
  Heeg-genus.pdf}.

\bibitem[SQ]{ScharlemannQiu}
M.~Scharlemann and R.~Qiu.
\newblock A proof of the {G}ordon {C}onjecture.
\newblock Preprint. Available at {\tt http://arxiv.org/abs/0801.4581}.

\bibitem[ST93]{st:93}
Martin Scharlemann and Abigail Thompson.
\newblock Heegaard splittings of {$({\rm surface})\times I$} are standard.
\newblock {\em Math. Ann.}, 295(3):549--564, 1993.

\bibitem[ST94]{st:94}
M.~Scharlemann and A.~Thompson.
\newblock Thin position for 3-manifolds.
\newblock {\em A.M.S. Contemporary Math.}, 164:231--238, 1994.

\bibitem[SW07]{SchultensWeidmann}
Jennifer Schultens and Richard Weidmann.
\newblock Destabilizing amalgamated {H}eegaard splittings.
\newblock In {\em Workshop on {H}eegaard {S}plittings}, volume~12 of {\em Geom.
  Topol. Monogr.}, pages 319--334. Geom. Topol. Publ., Coventry, 2007.

\end{thebibliography}

\end{document}